\documentclass[11pt]{article}

\usepackage[utf8]{inputenc}
\usepackage{epsfig}
\usepackage{amsmath}
\usepackage{amsfonts}
\usepackage{amsthm}
\usepackage{color}

        \theoremstyle{plain}
        
        \newtheorem{proposition}{Proposition}[section]

        \theoremstyle{remark}
        
        \theoremstyle{remark}
        
        \theoremstyle{remark}

\newcommand{\demi}{\frac{1}{2}}

\newcommand{\tdemi}{\textstyle \frac{1}{2}}
\newcommand{\troisdemi}{\textstyle \frac{3}{2}}

\newcommand{\R}{{\mathbb{R}}}

\newcommand{\dt}{\partial_t}

\newcommand{\divv}{\nabla_v \cdot}

\newcommand{\D}{D}
\renewcommand{\L}{L}
\newcommand{\cint}[1]{\left\langle #1\right\rangle}

\newcommand{\eps}{\varepsilon}
\renewcommand{\Pr}{{\rm Pr}}

\begin{document}


\begin{center}

{\bf A Fokker-Planck model of the Boltzmann equation with correct
  Prandtl number }

\vspace{1cm}
J. Mathiaud$^1$,  L. Mieussens$^2$

\bigskip
$^1$CEA-CESTA\\
15 avenue des sabli\`eres - CS 60001\\
33116 Le Barp Cedex, France\\
{ \tt(julien.mathiaud@cea.fr)}\\
Fax number: (33)557045433\\
Phone number: (33)557046872

\bigskip
$^2$Univ. Bordeaux, IMB, UMR 5251, F-33400 Talence, France.\\
CNRS, IMB, UMR 5251, F-33400 Talence, France. \\
INRIA, F-33400 Talence, France. \\

{ \tt(Luc.Mieussens@math.u-bordeaux1.fr)}

\end{center}

Abstract: We propose an extension of the Fokker-Planck model of the
Boltzmann equation to get a correct Prandtl number in the Compressible
Navier-Stokes asymptotics. This is obtained by replacing the diffusion
coefficient (which is the equilibrium temperature) by a non diagonal
temperature tensor, like the Ellipsoidal-Statistical model (ES) is
obtained from the Bathnagar-Gross-Krook model (BGK) of the Boltzmann
equation. Our model is proved to satisfy the properties of
conservation and a H-theorem. A Chapman-Enskog analysis and two
numerical tests show that a correct Prandtl number of $\frac23$ can be
obtained.

\bigskip


\section{Introduction}

Numerical simulations of rarefied gas flows are fundamental tools to
study the behavior of a gas in a system which length is of same order
of magnitude as the mean free path of the gas molecules. For instance,
these simulations are used in aerodynamics to estimate the heat flux
at the wall of a re-entry space vehicle at high altitudes; numerical
simulations are also used to estimate the attenuation of a
micro-accelerometer by the surrounding gas in a
micro-electro-mechanical system; a last example is when one wants to
estimate the pumping speed or compression rate of a turbo-molecular pump. 

Depending on the problem, the flow can be in non-collisional,
rarefied, or transitional regime, and several simulation methods
exist. The most common method is the Direct Simulation Monte Carlo
method (DSMC) proposed by Bird~\cite{bird}: it is a stochastic method that simulate the
flow with macro particles that mimic the real particles with transport
and collisions. It works well for rarefied flows, but can be less
efficient what the flow is close to continuous regime (even if there
are other version of DSMC that are designed to work well in this
regime: see for instance~\cite{RHW_2011,DP_2014}). For
non-collisional regimes, the Test Particle Monte Carlo (TPMC) method
is very efficient and is a common tool for vacuum pumps~\cite{Davis_TPMC}. For
transition regimes, there exist mature deterministic solvers based on a direct
approximation of the Boltzmann equation that can be
used~\cite{Titarev_2012,BCHM_2014,CA_2011,CXLC_2012,KAAFZ,WRZ_2014b}. We refer to~\cite{luc_rgd2014} or further references on
these solvers. Among these solvers, some of them are based on model
kinetic equations (BGK, ES-BGK, Shakhov): while these models keep the
accuracy of the Boltzmann equation in the transition regime, they are
much simpler to be discretized and give very efficient
solvers~\cite{LZ_2009,CA_2011,Titarev_2012,CXLC_2012,BCHM_2014}. These models simply replace the
Boltzmann collision operator by a relaxation operator toward some
equilibrium distribution, but still retain the elementary properties
of the Boltzmann operator (collisional invariants, Maxwellian equilibrium). Two of them (BGK
and ES-BGK) also satisfy the important entropy dissipation property
(the so-called ``H-theorem'')~\cite{ALPP}.

Very recently, in a series of
paper~\cite{Jnny2010,Grj2011,Grj2012,Grj2013}, Jenny et al. proposed
a very different and innovative approach: they proposed to use a
different model equation known as the Fokker-Planck model to design a
rarefied flow solver. In this model, the collisions are taken into
account by a diffusion process in the velocity space. Like the model
equations mentioned above, this equation also satisfy the main
properties of the Boltzmann equation, even the H-theorem. Instead of
using a direct discretization of this equation, the authors used the
equivalent stochastic interpretation of this equation (the Langevin
equations for the position and velocity of particles) that are
discretized by a standard stochastic ordinary differential equation
numerical scheme. This approach turned out to be very efficient, in
particular in the transition regime, since it is shown to be
insensitive to the number of simulated particles, as opposed to the
standard DSMC. However, this Fokker-Planck model is parametrized by a
single parameter, like the BGK model, and hence cannot give the
correct transport coefficients in near equilibrium regimes: this is
often stated by showing that the Prandtl number--which the ration of
the viscosity to the heat transfer coefficient--has an incorrect
value. The authors have proposed a modified model that allow to fit
the correct value of the Prandtl number~\cite{Grj2011}, and then have
extended it to more complex flows (multi-species~\cite{Grj2012} and
diatomic~\cite{Grj2013}). However, even if the results obtained with
this approach seem very accurate, it is not clear that these models
still satisfy the H-theorem.

In this paper, we propose another kind of modification of
the Fokker-Planck equation to get a correct Prandtl number: roughly
speaking, the diffusion coefficient (which is the equilibrium
temperature) is replaced by a non diagonal temperature tensor. This
approach it is closely related to the way the BGK model is extended to
the ES-BGK model, and we call our model the ES-Fokker Planck (ES-FP)
model. Then, we are able to prove that this model satisfies the
H-theorem. For illustration, we also show numerical experiments that
confirm our analysis, for a space homogeneous problem.

The outline of this paper is the following. In
section~\ref{sec:FPmodel}, our model and is properties are proved,
including the H-theorem. A Chapman-Enskog analysis is made in
section~\ref{sec:CE} to compute the transport coefficient obtained with
this model. In section~\ref{sec:varnu}, we show that our model can
include a variable parameter that allow to fit the correct Prandtl
number. Our numerical tests are shown in section~\ref{sec:num}, and
conclusions and perspectives are presented in section~\ref{sec:concl}.

\section{The Fokker-Planck model}
\label{sec:FPmodel}

   \subsection{The Boltzmann equation}
   We consider a gas described by the mass density of particles
   $f(t,x,v)$ that at time $t$ have the position $x$ and the velocity
   $v$ (note that both position $x$ and velocity $v$ are scalar). The
   corresponding macroscopic quantities are $(\rho,\rho u, E) =
   \cint{(1,v,\tdemi |v|^2) f}$, where $\rho$, $\rho u$, and $E$ are
   the mass, momentum, and energy densities, and $\cint{\phi}=\int_{\R^3}
   \phi(v)\, dv$ for any velocity dependent function. The temperature
   $T$ of the gas is defined by relation $E=\tdemi \rho |u|^2+\frac{3}{2}
   \rho R T$, where $R$ is the gas constant, and the pressure is
   $p=\rho R T$.

The evolution of the gas is governed by the following Boltzmann equation
\begin{equation}\label{Bo}
 \dt f+v\cdot\nabla_x f = Q(f,f),
\end{equation}
with
$$ Q(f,f)(v) = \int_{v_*\in\R^3}\int_{\sigma\in S^2} \bigg( f(v'_*)\,f(v')- f(v_*)\,f(v) \bigg) \, r^2\, |v-v_*|\, d\sigma\,dv_* , $$
and
$$ v' = \frac{v+v_*}2 + \frac{|v-v_*|}2\,\sigma, $$
$$ v'_* = \frac{v+v_*}2 - \frac{|v-v_*|}2\,\sigma. $$

It is well known that this operator conserves the mass, momentum, and
energy, and that the local entropy $H(f)=\cint{f\log f}$ is locally
non-increasing. This means that the effect of this operator is to make 
the distribution $f$ relax towards its own local
Maxwellian distribution, which is defined by
\begin{equation*}
  M(f)(v)=\frac{\rho}{\left(2\pi RT\right)^{3/2}}\exp\left(\frac{|v-u|^{2}}{2RT}\right).
\end{equation*}

For the sequel, it is useful to define another macroscopic quantity,
which is not conserved: the temperature tensor, defined by 
\begin{equation}\label{eq-defTheta} 
   \Theta:=\frac1\rho\cint{(v-u)\otimes(v-u)f}.
\end{equation}
In an equilibrium state (that is to say when $f=M(f)$), $\Theta$
reduces to the isotropic tensor $RT I$.

   \subsection{The ES-Fokker Planck model}
The standard Fokker-Planck model for the Boltzmann equation is
\begin{equation}  \label{eq-FP}
\dt f+v\cdot\nabla_x f = \frac{1}{\tau}\divv\bigl((v-u)f+RT\nabla_v f\bigr),
\end{equation}
see~\cite{cercignani}. Our model is obtained in the same spirit as the
ES model is obtained from a modification of the BGK equation: the
temperature that appears in~(\ref{eq-FP}), as a diffusion coefficient, 
is replaced by a tensor $\Pi$ so that we obtain
\begin{equation}  \label{eq-es_fpl}
\dt f+v\cdot\nabla_x f = \D(f),
\end{equation}
where the collision operator is defined by
\begin{equation}  \label{eq-D}
D(f)=\frac{1}{\tau}\divv\bigl((v-u)f+{\Pi}\nabla_v f\bigr),
\end{equation}
where $\tau$ is a relaxation time, and $\Pi$ is a convex combination between the temperature tensor
$\Theta$ and its equilibrium value $RT I$, that is to say: 
\begin{equation}\label{eq-defPi} 
  \Pi=(1-\nu)RT I+\nu\Theta,
\end{equation}
with $\nu$ a parameter. According to~\cite{ALPP}, $\Pi$ is symmetric
positive definite if $\nu\in]-\demi,1]$. In fact, this condition is
too restrictive and we have the following result.
\begin{proposition} [Condition of definite positiveness of $\Pi$] \label{prop:defpos}

\

  The tensor $\Pi$ is symmetric positive definite for every tensor
  $\Theta$ if, and only if, 
\begin{equation} 
  -\frac{RT}{\lambda_{max}-RT}< \nu < \frac{RT}{RT-\lambda_{min}}, \label{positif}
\end{equation}
where $\lambda_{max}$ and $\lambda_{min}$ are the (positive) maximum and minimum
eigenvalues of $\Theta$.
Moreover $\Pi$  is positive definite independently of the eigenvalues of $\Theta$ as long as:
\begin{equation} 
  -\frac{1}{2}< \nu < 1, \label{positif2}
\end{equation}
\end{proposition}
\begin{proof}
  Since $\Theta$ is symmetric and positive definite, it is
  diagonalizable and its eigenvalues are positive. Moreover, because of the
  definition of $\Pi$, both $\Theta$ and $\Pi$ are diagonalizable in
  the same orthogonal basis. Then there exists an invertible
  orthogonal matrix $P$ such that
$$P\Theta P^{T}=\left(\begin{array}{ccc}
\lambda_1&0&0\\
0&\lambda_2&0\\
0&0&\lambda_3\\
\end{array}\right),$$ 
and
$$P\Pi P^{T}=\left(\begin{array}{ccc}
(1-\nu)RT+\nu\lambda_1&0&0\\
0&(1-\nu)RT+\nu\lambda_2&0\\
0&0&(1-\nu)RT+\nu\lambda_3\\
\end{array}\right).
$$ 
Clearly $\Pi$ is  positive definite if and only if $\,
(1-\nu)RT+\nu\lambda_i>0$ for every $i$, which leads to:
\begin{equation*} 
  -\frac{RT}{\lambda_{max}-RT}< \nu < \frac{RT}{RT-\lambda_{min}},
\end{equation*}
which is~(\ref{positif}).
At worse $\lambda_{min}$ is equal to zero so that  
\begin{equation*} 
  \nu < \frac{RT}{RT-0}=1.
\end{equation*}
At worse $\lambda_{max}$ is equal to $Trace(\Theta)=3RT$ so that  
\begin{equation*} 
  \nu > -\frac{RT}{3RT-RT}=-\frac{1}{2}.
\end{equation*}
On the contrary when $\lambda_{max}$  tends to $RT$ near Maxwellian equilibrium  (note that $\lambda_{min} $ also tends to $T$  in this case), $\nu$ becomes unbounded.
The last two inequalities provide inequalities~(\ref{positif2}).

\end{proof}

\begin{proposition} \label{prop:formulations}
  The operator $\D$ has two other equivalent formulations:
\begin{equation}  \label{eq-D1}
\D(f) = \frac{1}{\tau}\divv \left( \Pi G(f)\nabla_v \frac{f}{G(f)} \right),
\end{equation}
and
\begin{equation}  \label{eq-D2}
\D(f) = \frac{1}{\tau}\divv \left( \Pi f\nabla_v \log \left( \frac{f}{G(f)}\right) \right),
\end{equation}
where $G(f)$ is the anisotropic Gaussian defined by
\begin{equation}  \label{eq-G}
  G(f)=\frac{\rho}{\sqrt{\det(2\pi\Pi)}}\exp\left(-\frac{(v-u)\Pi^{-1}(v-u)}{2}\right),
\end{equation}
which has the same 5 first moments as $f$
\begin{equation*}
  \cint{(1,v,\tdemi|v|^2)G(f)} = (\rho,\rho u, E),
\end{equation*}
and has the temperature tensor $\cint{(v-u)\otimes(v-u)G(f)}=\Pi$.
\end{proposition}

\begin{proof}
Relations~(\ref{eq-D1}) and~(\ref{eq-D2}) are simple consequences of the relation $\nabla_v G(f)
= -\Pi^{-1}(v-u)G(f)$. The moments of $G(f)$ are given by standard
Gaussian integrals.
\end{proof}

Now, we state that $D$ has the same conservation and entropy
properties as the Boltzmann collision operator $Q$. 
\begin{proposition} \label{prop:propD}
We assume $\nu$ satisfies~(\ref{positif}) and that $\nu<1$.
  The operator $D$ conserves the mass, momentum, and energy: 
\begin{equation}\label{eq-consD} 
  \cint{(1,v,\tdemi|v|^2)\D(f)} = 0,
\end{equation}
it satisfies the dissipation of the entropy:
\begin{equation*}
  \cint{D(f)\log f}\leq 0,
\end{equation*}
and we have the equilibrium property:
\begin{equation*}
  \D(f) = 0 \Leftrightarrow f= G(f) \Leftrightarrow f=  M(f).
\end{equation*}
\end{proposition}
\begin{proof}
For the conservation property, we take any function $\phi(v)$ and we
compute
\begin{equation*}
  \cint{D(f)\phi}
 = -\frac{1}{\tau}\cint{((v-u)f + \Pi \nabla_vf)\cdot \nabla_v \phi}.
\end{equation*}
With $\phi(v)=1$, we directly get $\cint{D(f)}=0$. With $\phi(v)=v_i$,
we get 
\begin{equation*}
  \cint{D(f)v_i}
 = -\frac{1}{\tau}\cint{(v_i-u_i)f + \Pi_{i,j} \partial_{v_j}f} = 0,
\end{equation*}
since $\cint{v_i f} = \rho u_i = \rho \cint{v_i f}$ and
$\cint{\partial_{v_j}f} = 0$.
For $\phi(v)=\demi|v|^2$, we get
\begin{equation*}
  \cint{D(f)v_i}
 = -\frac{1}{\tau}\cint{((v_i-u_i)f + \Pi_{i,j} \partial_{v_j}f)v_i}
= -\frac{1}{\tau}(3\rho RT - \rho {\rm Trace(\Pi))}= 0,
\end{equation*}
since~(\ref{eq-defPi}) shows that ${\rm Trace}(\Pi) = (1-\nu)3RT + \nu
{\rm Trace}(\Theta)$ and~(\ref{eq-defTheta}) implies ${\rm
  Trace}(\Theta) = 3RT$.

\

For the entropy we need some inequalities to proceed. As noticed in
the proof of proposition~\ref{prop:defpos}, $\Theta$ and $\Pi$ are
symmetric and can be diagonalized in the same orthogonal basis. 
Like it is done in \cite{ABLP}, we will now suppose that we work in this basis so that:\\ 
$$\Theta=\left(\begin{array}{ccc}
\lambda_1&0&0\\
0&\lambda_2&0\\
0&0&\lambda_3\\
\end{array}\right),$$ 
and
$$\Pi=\left(\begin{array}{ccc}
(1-\nu)RT+\nu\lambda_1&0&0\\
0&(1-\nu)RT+\nu\lambda_2&0\\
0&0&(1-\nu)RT+\nu\lambda_3\\
\end{array}\right).
$$ 
For the following it is important to remind that $\sum_{i=1}^3
\lambda_i=3RT$ and that all eigenvalues $\lambda_i$  are strictly positive.

Since we want to prove that $\cint{D(f)\ln f}\leq 0$, we compute
$\cint{D(f)\ln f}$ with an integration by parts to get
\begin{eqnarray}
  \cint{D(f)\ln(f)}&=&-\sum_{i=1}^3\cint{(\left(v_i- u_i\right)
    f+\left((1-\nu)RT+\nu\lambda_i\right)\partial_i
    f)\left(\frac{\partial_i f}{f}\right)} \nonumber \\
  &=&-\sum_{i=1}^3\frac{(1-\nu)RT+\nu\lambda_i}{\lambda_i}\cint{(\left(v_i-u_i\right) f+\lambda_i\partial_i f)\left(\frac{\partial_i f}{f}\right)} \nonumber\\
  &&-\sum_{i=1}^3\left(1-\frac{(1-\nu)RT+\nu\lambda_i}{\lambda_i}\right)\cint{(\left(v_i-u_i\right) f\left(\frac{\partial_i f}{f}\right)} \nonumber\\
  &=&-\sum_{i=1}^3\frac{(1-\nu)RT+\nu\lambda_i}{\lambda_i}\cint{(\left(v_i-u_i\right) f+\lambda_i\partial_i f)\left(\frac{\partial_i f}{f}+\frac{ \left(v_i-u_i\right)}{\lambda_i}-\frac{ \left(v_i-u_i\right)}{\lambda_i}\right)} \nonumber\\
  &&-\sum_{i=1}^3\left(1-\frac{(1-\nu)RT+\nu\lambda_i}{\lambda_i}\right)\cint{(\left(v_i-u_i\right)
    f\left(\frac{\partial_i f}{f}\right) }\nonumber\\
  &=&A+B+C \nonumber,
\end{eqnarray}
\
with
\begin{eqnarray} 
&&A=-\sum_{i=1}^3\frac{(1-\nu)RT+\nu\lambda_i}{\lambda_i}\cint{(\left(v_i-u_i\right)
  f+\lambda_i\partial_i f)\left(\frac{\partial_i f}{f}+\frac{
      \left(v_i-u_i\right)}{\lambda_i}\right)}\label{eq-A},  \\
&&B=-\sum_{i=1}^3\frac{(1-\nu)RT+\nu\lambda_i}{\lambda_i}\cint{(\left(v_i-u_i\right)
  f+\lambda_i\partial_i f)\left(-\frac{
      \left(v_i-u_i\right)}{\lambda_i}\right)}\label{eq-B},  \\
&&C=-\sum_{i=1}^3\left(1-\frac{(1-\nu)T+\nu\lambda_i}{\lambda_i}\right)\cint{(\left(v_i-u_i\right)
  f\left(\frac{\partial_i f}{f}\right) }\label{eq-C} .
\end{eqnarray}

First, note that $A$ is clearly negative since
\begin{equation*}
  (\left(v_i-u_i\right) f+\lambda_i\partial_i
  f)\left(\frac{\partial_i f}{f}+\frac{
      \left(v_i-u_i\right)}{\lambda_i}\right)={f}{\lambda_i}
{\left(\frac{\partial_i f}{f}+\frac{ \left(v_i-u_i\right)}{\lambda_i}\right)}^2,
\end{equation*}
and $((1-\nu)RT+\nu\lambda_i)$ and $\lambda_i$ are positive (due to the
assumption on $\nu$ and proposition~\ref{prop:defpos}), and $f$ is
positive too. Then, $B$ is equal to $0$ since by integration by parts,
we get $\cint{((v_i-u_i)f +\lambda_i \partial_i f)\left(-\frac{
      \left(v_i-u_i\right)}{\lambda_i}\right)} =
-\frac{1}{\lambda_i}\cint{(v_i-u_i)^2f}+ \cint{f}$, which is 0 since
$\cint{(v_i-u_i)^2f}=\rho \lambda_i$.
 
 \
 
 The last term $C$ satisfies
 \begin{eqnarray*}
 C&=&-\sum_{i=1}^3\left(1-\frac{(1-\nu)RT+\nu\lambda_i}{\lambda_i}\right)\cint{\left(v_i-u_i\right) f\left(\frac{\partial_i f}{f}\right)}\\
 &=&-\sum_{i=1}^3\left(1-\frac{(1-\nu)RT+\nu\lambda_i}{\lambda_i}\right)\cint{\left(v_i-u_i\right)\partial_i f }\\
  &=&\sum_{i=1}^3\left(1-\frac{(1-\nu)RT+\nu\lambda_i}{\lambda_i}\right)\cint{f}\\
  &=&\sum_{i=1}^3 (1-\nu)\left(1-\frac{RT}{\lambda_i}\right)\cint{f}\\
  &=&\left(3-\sum_{i=1}^3\frac{RT}{\lambda_i} \right)(1-\nu)\cint{f}.
 \end{eqnarray*}



The Jensen inequality gives
$\frac{1}{3}\sum_{i=1}^3\frac{1}{\lambda_i}\geq
(\frac{1}{3}\sum_{i=1}^3\lambda_i)^{-1} = \frac{1}{RT}$,
which shows that $C$ is non positive (since $\nu<1$), and hence concludes the proof
of the entropy inequality. 

In the general case, $\Theta$ and $\Pi$ are not diagonal. However, the
same analysis can still be done: it is sufficient to use the change of
variables $v'=P^Tv$, where $P$ is the matrix of the orthonormal basis
of eigenvectors in which $\Theta$ and $\Pi$ are diagonal. Indeed, let
us define $f'(v')=f(v)$, then $\nabla_v f(v) = P \nabla_v' f'(v')$,
and
\begin{equation*}
  \int_{\R^3} (1,v',(v'-u')\otimes(v'-u')) f'(v') \, dv' = (\rho,\rho
  u',\rho \Theta'),
\end{equation*}
where $u'=P^Tu$, and where $\Theta'=P^T\Theta P$ is diagonal. We also
define $\Pi'=P^T\Pi P$ which is also diagonal ad we have
$\Pi'=(1-\nu)RT I + \nu \Theta'$. Then it is easy to find that we again
have $\cint{D(f)\ln(f)}=A+B+C$, but now with $A$, $B$, and $C$ that
have to be defined by~(\ref{eq-A})--(\ref{eq-C}) with prime variables
$f',v'$, and where $\cint{}$ denotes integration with respect to
$v'$. Then we are back to the diagonal case and the previous proof is
still valid.


For the equilibrium property, we use~(\ref{eq-D1}) to obtain
\begin{equation*}
 \cint{\D(f)\frac{f}{G(f)}} 
= -\frac{1}{\tau}\cint{ G(f)\left( \Pi \nabla_v \frac{f}{G(f)}
  \right)\cdot \nabla_v \frac{f}{G(f)}}\leq 0,
\end{equation*}
since $\Pi$ is symmetric positive definite.
Moreover,
this integral is zero if and only if $\nabla_v \frac{f}{G(f)} =0$,
which is equivalent to $f=\alpha G(f)$. Since $f$ and $G(f)$ have the
same first 5 moments, this relation is true if and only if $\alpha=1$,
and hence $f=G(f)$. Finally, this last relation implies
$\cint{(v-u)\otimes(v-u)f} = \cint{(v-u)\otimes(v-u)G(f)}$, that is to
say $\Theta=\Pi$. Then using~(\ref{eq-defPi}) gives $\Theta = RT I$
which implies $G(f) = M(f)$ and hence $f=M(f)$. Conversely, if
$f=M(f)$, it is obvious that $G(f) = M(f)$ and $D(f) = 0$.
\end{proof}

\section{Chapman-Enskog analysis}
\label{sec:CE}

In this section, we prove the following proposition.
\begin{proposition}
  The solution of the kinetic model~(\ref{eq-es_fpl_adim}) satisfies,
  up to $O(\eps^2)$, the Navier-Stokes equations
\begin{equation}\label{eq-ns} 
\begin{split}
& \dt \rho + \nabla \cdot \rho u = 0, \\
& \dt \rho u + \nabla\cdot  (\rho u\otimes u) + \nabla p = -\nabla
\cdot \sigma, \\
& \dt E + \nabla \cdot (E+p)u = -\nabla\cdot q - \nabla\cdot(\sigma u),
\end{split}
\end{equation}
where the shear stress tensor and the heat flux are given by
\begin{equation}  \label{eq-fluxes_ns}
\sigma = -\mu \bigl(\nabla u + (\nabla u)^T -\frac{2}{3}\nabla\cdot
u\bigr), \quad \text{and} \quad  q=-\kappa \nabla \cdot T,
\end{equation}
with the following values of the viscosity and heat
transfer coefficients
\begin{equation}  \label{eq-coef}
\mu = \frac{\tau p}{2(1-\nu)}, \quad \text{and} \quad 
\kappa = \frac{5}{6}\tau p R.
\end{equation}
Moreover,
the corresponding Prandtl number is 
\begin{equation*}
\Pr = \frac{3}{2(1-\nu)},
\end{equation*}
and $\eps$ is the Knudsen number defined below.
\end{proposition}

In the first step of the proof, we write the conservation laws
derived from~(\ref{eq-es_fpl_adim}) and~(\ref{eq-consD})
\begin{equation}\label{eq-lois_cons_bis} 
\begin{split}
& \dt \rho + \nabla_x\cdot \rho u = 0, \\
& \dt \rho u + \nabla_x\cdot (\rho u\otimes u) + \nabla_x\cdot \Sigma(f) = 0, \\
& \dt E + \nabla_x\cdot (Eu+\Sigma(f) u + q(f)) = 0,
\end{split}
\end{equation}
where $\Sigma(f)$ and $q(f)$ denote the stress tensor and the heat
flux, defined by
\begin{equation}  \label{eq-Sigma_q}
\Sigma(f)=\cint{(v-u)\otimes(v-u)f} \qquad q(f)=\cint{\tdemi (v-u)|v-u|^2f}.
\end{equation}
The Chapman-Enskog procedure consists in looking for an approximation
of $\Sigma(f)$ and $q(f)$ up to order $O(\eps^2)$ as functions of
$\rho,u,T$ and their gradients.

To do so, we now write our model in a
non-dimensional form. Assume we have some reference values of length
$x_*$, pressure $p_*$, and temperature $T_*$. With these reference
values, we can derive reference values for all the other quantities:
mass density $\rho_*=p_*/RT_*$, velocity $v_*=\sqrt{RT_*}$, time
$t_*=x_*/v_*$, distribution function $f_*=\rho_*/(RT_*)^{3/2}$. We also
assume we have a reference value for the relaxation time $\tau_*$. By
using the non-dimensional variables $w'=w/w_*$ (where $w$ stands for
any variables of the problem), our model can be written
\begin{equation}  \label{eq-es_fpl_adim}
\dt f+v\cdot\nabla_x f = \frac{1}{\eps}\D(f),
\end{equation}
where $\eps = \frac{v_*\tau_*}{x_*}$ is the Knudsen number. Note that
since we always work with the non-dimensional variables from now
on, these variables are not written with the '
in~(\ref{eq-es_fpl_adim}).

Note that an important consequence of the use of these non-dimensional
variables is that $RT$ has to be replaced by $T$ in every expressions
given before. Namely, now $\Pi$ is defined by
\begin{equation}\label{eq-Pi_adim} 
  \Pi=(1-\nu)T I+\nu\Theta,
\end{equation}
the temperature is now defined by
\begin{equation*}
E=\tdemi \rho |u|^2+\troisdemi \rho T,
\end{equation*}
and the Maxwellian of $f$ now is
\begin{equation*}
  M(f)=\frac{\rho}{\left(2\pi
      T\right)^{3/2}}\exp\left(\frac{|v-u|^{2}}{2T}\right).
\end{equation*}

Now, it is standard to look for the deviation of $f$ from its own
local equilibrium, that is to say to set $f=M(f)(1+\eps g)$. However,
this requires the linearization of the collision operator $D$ around
$M(f)$, which is not very easy. At the contrary, it will be shown that
it is much simpler to look for the deviation of $f$ from the
anisotropic Gaussian $G(f)$ defined in~(\ref{eq-G}). Since it can
easily be seen that $M(f)$ and $G(f)$ are close up to $O(\eps)$ terms,
this expansion is sufficient to get the Navier-Stokes equations.

First, we give some approximation properties.
\begin{proposition}
  We write $f$ as $f=G(f)(1+\eps g)$. Then we have:
\begin{equation}\label{eq-mtsg} 
  \cint{(1,v,\tdemi |v|^2)G(f)g} = 0.
\end{equation} 
Moreover, if we assume that the deviation
  $g$ is an $O(1)$ with respect to $\eps$, then we have
\begin{equation}\label{eq-approxPi} 
  \Pi=TI + O(\eps)
\end{equation}
and 
\begin{equation}  \label{eq-approxG}
G(f) = M(f) + O(\eps).
\end{equation}
\end{proposition}
\begin{proof}
We have already noticed that
$\cint{(1,v,\tdemi|v|^2)G(f)}=\cint{(1,v,\tdemi|v|^2)f}$ (see proposition~\ref{prop:formulations}), which
gives~(\ref{eq-mtsg}). The other relations~(\ref{eq-approxPi})
and~(\ref{eq-approxG}) are obtained by simple Taylor expansions
of~(\ref{eq-Pi_adim}) and~(\ref{eq-G}) with
respect to $\eps$.
\end{proof}

Now, we show that $\Sigma(f)$ and $q(f)$ can be approximated up to
$O(\eps^2)$ by using the previous decomposition.
\begin{proposition} \label{prop:approx_Sigma}
  We have
\begin{equation}  \label{eq-approx_Sigma}
\Sigma(f) = p I + \frac{\eps}{1-\nu}\Sigma(M(f)g) + O(\eps^2),
\end{equation}
and
\begin{equation}  \label{eq-approx_q}
q(f) = \eps q(M(f)g) + O(\eps^2).
\end{equation}
\end{proposition}
\begin{proof}
  By using $f=G(f)(1+\eps g)$, we have 
\begin{equation}  \label{eq-Sigmaf1}
\Sigma(f) = \Sigma(G(f)) +
  \eps \Sigma(G(f)g) = \rho \Pi + \eps \Sigma(G(f)g) .
\end{equation}
But since by
  definition $\Sigma(f)= \rho \Theta(f)$, then~(\ref{eq-Pi_adim}) also
  implies $\rho \Pi = (1-\nu)\rho T I + \nu \Sigma(f)$. Using this
  relation in~(\ref{eq-Sigmaf1}) gives
$\Sigma(f) =  (1-\nu)\rho T I + \nu \Sigma(f) + \eps \Sigma(G(f)g)$
which yields~(\ref{eq-approx_Sigma}). The
approximation~(\ref{eq-approx_q}) is immediately deduced from the
decomposition $f=G(f)(1+\eps g)$.

\end{proof}

Consequently, the Navier-Stokes equations can be obtained provided
that $\Sigma(M(f)g)$ and $q(M(f)g)$ can be approximated, up to
$O(\eps)$. This can be done by looking for an approximation of
the deviation $g$ itself: it is obtained by using the decomposition
$f=G(f)(1+\eps g)$ into~(\ref{eq-es_fpl_adim}), which gives
\begin{equation*}
  \dt G(f)+v\cdot\nabla_x G(f) + O(\eps) = \frac{1}{\eps}\D(G(f)(1+\eps
g)).
\end{equation*}
First, we state what the expansion of
$\D(G(f)(1+\eps g))$ is. 
\begin{proposition}
We have
\begin{equation*}
  D(G(f)(1+\eps
g)) = \eps \frac{1}{\tau}M(f) \L g + O(\eps^2), 
\end{equation*}
where $\L$ is the linear operator
\begin{equation}  \label{eq-L}
\L g = \frac{1}{M(f)}\divv (TM(f)\nabla_v g).
\end{equation}
\end{proposition}
\begin{proof}
Formulation~(\ref{eq-D1}) of $D(f)$ gives
\begin{equation}\label{eq-Dexpansion} 
  \D(f) = \frac{1}{\tau}\divv\left( \Pi G(f) \nabla_v
    \frac{G(f)(1+\eps g )}{G(f)}\right).
\end{equation}
A direct computation gives $\nabla_v \frac{G(f)(1+\eps g )}{G(f)} =
\eps\nabla_v g$. Using this relation and~(\ref{eq-approxPi})
into~(\ref{eq-Dexpansion}) gives $ D(G(f)(1+\eps g)) =
\frac{\eps}{\tau}\divv (TG(f)\nabla_v g)
+ O(\eps^2)$. Using~(\ref{eq-approxG}) gives the final result.
\end{proof}
This proposition shows us that $g$ satisfies $M(f)\L g  = \tau(\dt G(f) + v
\cdot G(f)) + O(\eps) $, and then using~(\ref{eq-approxG}) gives
\begin{equation} \label{eq-Lg}
  \L g  = \frac{\tau}{M(f)}(\dt M(f) + v \cdot \nabla_x M(f)) + O(\eps) .
\end{equation}
We now have to solve this equation approximately to obtain an approximation of $g$.

First, the right-hand side of~(\ref{eq-Lg}) is expanded, and the time
derivatives of $\rho, u,$ and $T$ are approximated up to $O(\eps)$ by
their space gradients by using~(\ref{eq-lois_cons_bis}). Then we get
\begin{equation*}
  \frac{1}{M(f)}(\dt M(f) + v \cdot \nabla_x M(f)) = A(V) \cdot \frac{\nabla
    T}{\sqrt{T}} + B(V) : \nabla u + O(\eps),
\end{equation*}
where $V=(v-u)/\sqrt{T}$ and
\begin{equation*}
  A(V) = \left(\frac{|V|^2}{2}-\frac{5}{2}\right)V 
\quad \text{ and } \quad  B(V) = V\otimes V - \frac{1}{3}|V|^2I.
\end{equation*}
Then~(\ref{eq-Lg}) reads
\begin{equation} \label{eq-Lgbis}
  \L g  = \tau \left(A(V) \cdot \frac{\nabla T}{\sqrt{T}} + B(V) : \nabla u\right) + O(\eps).
\end{equation}
This equation can be solved by showing that $A$ and $B$ are
eigenvectors of $\L$. Indeed, we have the following proposition.
\begin{proposition}
  The components of $A$ and $B$ satisfy
\begin{equation*}
  \L A_i = - 3 A_i \quad \text{ and } \quad \L B_{i,j} = -2 B_{i,j}.
\end{equation*}
\end{proposition}
\begin{proof}
By using the change of variables $V=(v-u)/\sqrt{T}$, the
definition~(\ref{eq-L}) reduces to
\begin{equation}\label{eq-LV} 
  Lg = \nabla_V \cdot (M_0(V)\nabla_V g),
\end{equation}
where $M_0(V) =
\frac{1}{(2\pi)^{\frac{3}{2}}}\exp(-\frac{|V|^2}{2})$. The direct
  computation of $\L A_i$ and $\L B_{i,j}$ is easily obtained
  from~(\ref{eq-LV}) and is left to the reader.
\end{proof}
This property shows that we can look for an approximate solution
of~(\ref{eq-Lgbis}) as a linear combination of the components of $A$
and $B$. We find
\begin{equation*}
  g = -\frac{\tau}{3}\frac{\nabla T}{\sqrt{T}} \cdot A(V) -
  \frac{\tau}{2}\nabla u : B(V) + O(\eps).
\end{equation*}
Then we just have to insert this expression into $\Sigma(M(f)g)$ and
$q(M(f)g)$ to get, after calculation of Gaussian integrals,
\begin{equation*}
\begin{split}
& \Sigma(M(f)g) =  -\frac{\tau p }{2} (\nabla u + (\nabla u)^T -\frac{2}{3}\nabla
\cdot u I) + O(\eps) \\
& q(M(f)g) =  - \frac{\tau p }{3} \nabla T + O(\eps).
\end{split}
\end{equation*}

Now we use proposition~\ref{prop:approx_Sigma} and we come back to the
dimensional variables to get
\begin{equation*} 
\Sigma(f) = p I - \mu \bigl(\nabla u + (\nabla u)^T -\frac{2}{3}\nabla\cdot
u\bigr), \quad \text{and} \quad  q=-\kappa \nabla \cdot T,
\end{equation*}
with the following values of the viscosity and heat
transfer coefficients
\begin{equation*}
\mu = \frac{\tau p}{2(1-\nu)}, \quad \text{and} \quad 
\kappa = \frac{5}{6}\tau p R.
\end{equation*}
Using these relations into the conservation
laws~\eqref{eq-lois_cons_bis} proves that $\rho,\rho u$, and $E$ solve
the Navier-Stokes equations~(\ref{eq-ns}) up to $O(\eps^2)$ with
transport coefficients given by~(\ref{eq-coef}) and that the
Prandtl number is
\begin{equation*}
  \Pr = \frac{\mu}{\kappa}\frac{5R}{2} = \frac{3}{2(1-\nu)}.
\end{equation*}

\section{The ES-Fokker Planck model with a non constant $\nu$}
\label{sec:varnu}

This short section establishes how we deal with the Prandtl number in all cases.
 
\subsection{Some limits of the model with a constant $\nu$}

In the previous section, we have found that the Prandtl number
obtained for the ES-Fokker Planck model is
\begin{equation}\label{eq-Pr} 
  \Pr = \frac{\mu}{\kappa}\frac{5R}{2} = \frac{3}{2(1-\nu)},
\end{equation}
and can be adjusted to various values by choosing a corresponding value
of the parameter $\nu$.

Moreover, we have seen that the model is well defined (that is to say
that the tensor $\Pi$ is positive definite for every $f$) if, and only
if $-\frac{1}{2}< \nu < 1$. This last condition leads to the following limitations for the Prandtl number:
$$1<Pr<+\infty,$$
so that the correct Prandtl number for monoatomic gases (which is
equal to $\frac23$) cannot be obtained. In the next section, we show
that this analysis, which is based on the inequality~(\ref{positif2}) is too
restrictive, and that there is a simple way to adjust the correct
Prandtl number.

\subsection{Recovering the good Prandtl number} 

The previous analysis relies on inequality~(\ref{positif2}) of
proposition~\ref{prop:defpos} that does not take into account the
distribution $f$ itself: this inequality ensures the positive
definiteness of $\Pi$ independently of $f$. However,
proposition~\ref{prop:defpos} also indicates that $\Pi$ is positive
definite for a given $f$ if $\nu$ satisfies~(\ref{positif}). This
inequality is less restrictive, since it depends on $f$ via the
temperature $T$ and the extremal eigenvalues of $\Theta$.

Now, our first idea is that $\nu$ can be set to a non constant
value (it may depend on time and space): it just has to lie in the interval
$[-\frac{RT}{\lambda_{max}-RT},1]$, 
so that 
it satisfies the assumptions of proposition~\ref{prop:propD}.
The second idea is that the Prandtl number makes sense when the flow is
close to the equilibrium, that is to say when $f$ is close to its own
local Maxwellian $M(f)$: in such case, $\Theta = RT I + O(\eps)$, and
all the eigenvalues of $\Theta$ are close to $RT$ up to $O(\eps)$
terms, which implies $\lambda_{max} = RT + C\eps$
. Consequently, the value of $\nu$ now lies
in 
$[ -\frac{RT}{RC\eps}, 1]$ 
which shows that $\nu$
can take any arbitrary value between $-\infty$ and $1$ when
$\eps$ is small enough, and in particular the value $\nu=-\frac54$
that gives the correct Prandtl number $\Pr=\frac23$ can be used. 

In other words, by defining $\nu$ as a non constant value that
satisfies~(\ref{positif}) and is lower than 1, we can adjust the
correct Prandtl number provided that $\nu$ can be set to $-\frac54$
near the equilibrium regime.

This analysis suggests a very simple definition of $\nu$ : we propose
to use the smallest negative $\nu$ such that:
\begin{itemize}
\item  $\Pi$ remains strictly definite positive,
\item $\nu\geq -\frac{5}{4}$.
\end{itemize}
This leads to the following definition:
\begin{equation}  \label{eq-nu}
\nu=\max\left(-\frac54, -\frac{RT}{\lambda_{max}-RT}\right).
\end{equation}

\

Note that in most realistic  cases $\nu$ is equal to
$-\frac54$. Indeed, $\nu\neq -\frac54$ implies $\lambda_{max}>1.8 RT$
which can only happen in case of highly non equilibrium flow with
strong directional non isotropy: such cases are very specific and
are usually not observed in aerodynamical flows, for instance.

\section{Numerical tests}
\label{sec:num}

We have chosen to present a stochastic method to implement the model,
since it is well suited to diffusion operators~\cite{LBP} (this was
also used in~\cite{Jnny2010}). In this
paper we only present results in homogeneous cases to understand how
it works. Further studies  will be done in inhomogeneous cases and
comparisons with full Boltzmann equation and/or ESBGK model will be
provided in forthcoming paper. In this section we solve
\begin{equation}  \label{eq-es_fph}
\dt f = \frac{1}{\tau}\divv\bigl((v-u)f+{\Pi}\nabla_v f\bigr).
\end{equation}

Our goal is to show numerically that we are able to capture the
correct Prandtl number for monoatomic gases. This numerical
illustration is based on the following remarks. First, the density,
velocity, and temperature of $f$ are constant in time: this is due to
the conservation properties of the collision operator. Second, the
heat flux  $q$ satisfies
\begin{equation}
\frac d{dt} q=-\frac 3{\tau}q,
\end{equation}
so that 
\begin{equation} \label{eqqh}
q(t)=q(0)\exp\left(-\frac{3t}\tau\right). 
\end{equation}
Finally, the tensor $\Theta$ satisfies the relaxation equation
\begin{equation}\label{eq-edotheta} 
\frac d{dt} \Theta=\frac 1{\tau}2(1-\nu)\left(RTI -\Theta\right),
\end{equation}
which shows that it tends to $RTI$ for large times, and hence $\nu$
tends to the constant value $-5/4$ (see~(\ref{eq-nu})). Then for large times $t>s$, $\nu$
can be assumed to be constant, and~(\ref{eq-edotheta}) can be solved
to give
\begin{equation} \label{eqtetah}
\Theta(t)=\exp\left(-\frac{2(1-\nu)(t-s)}{\tau}\right)\Theta(s)+\left(1-\exp\left(-\frac{2(1-\nu)(t-s)}{\tau}\right)\right)RTI.
\end{equation}

Using equations ( \ref{eqtetah}) and (\ref{eqqh}), one gets 
\begin{equation}
\frac{\ln\left(\left|\Theta_{i,i}(t)-RT\right|\right)} {\ln(|q|}=\frac{2(1-\nu)}{3}=\frac{1}{Pr}
\end{equation}
for $i=1,2,3$, so that we are able to recover the Prandtl number by
looking at the long time values of $\Theta(t)$ and $q(t)$.

   \subsection{A stochastic numerical method}

   We use a DSMC method to solve the problem. For homogeneous cases,
   the probability density function is approximated with $N$ numerical
   particles so that
$$f(t,v)\simeq \alpha \sum_i^N \omega_i \delta_{V_i(t)},$$
where $\omega_i$ is the numerical weight of numerical particle $i$ and
$\delta_{V_i(t)}$ is the Dirac function at the particle velocity
$V_i(t)$. Moreover $\alpha$ is defined through the constant density
$\rho=\int_{\mathbb{R}^3}f(t,v)dv$ by
$$\rho=\alpha \sum_i^N \omega_i.$$

In the test cases presented here, all numerical weights $\omega_i$ are
equal, and we just have to define the dynamics of the numerical
particles. To solve diffusive problems, it is well-known that using
Brownian motion is a good way to proceed (see \cite{LBP}): the
corresponding stochastic ordinary differential equation is called the
Ornstein-Uhlenbeck process that reads
\begin{equation}
dV_i(t)=-\frac{1}{\tau }\left(V_i(t)-u\right)+AdB(t)
\end{equation}
for each $1\leq i\leq N$. The quantity $dB(t)$ is a three dimensional
Brownian process. The matrix $A$ has to satisfy $AA^T =
\Pi$. Several choices are available for $A$.  The obvious one would be
to use the square root of $\Pi$ (which is a positive definite matrix):
this requires to compute the eigenvectors of $\Pi$ and may lead to
expensive computations.  We find it simpler to use the Choleski
decomposition because of the simplicity of the algorithm.  For the
time discretization we use a backward Euler method. The complete
algorithm is the following:

\begin{enumerate}
\item Approximate the initial data $f(0,v)$ by $\sum_i^N \omega
  \delta_{V_i(t)}$, where $N$ is the number of numerical particles and
  $\omega$ a constant numerical weight. The velocities are chosen
  randomly according to the particle density function $f(0,.)$.
\item Compute the tensor $\Theta$
\item Compute the three real eigenvalues the positive definite tensor
  $\Theta$ with Cardan's formula.
\item Compute the Choleski factorization $\Pi=A^TA$ of $\Pi$.
\item For $i$ from 1 to $N$, advance the velocity $V_i^n$ through the
  process:
\begin{equation*}
  V_i^{n+1}=\left(1-\frac {\Delta t}
    \tau\right)(V_i^{n}-u)+\sqrt{\frac {2\Delta t}
    \tau}A\left(\begin{array}{c}
      B_1\\
      B_2\\
      B_3
\end{array}\right)
\end{equation*}
where $\Delta t=t^{n+1}-t^{n}$, $B_1,B_2,B_3$ are random numbers
chosen through a standard normal law. Since the scheme is explicit we
enforce $\frac{\Delta t}{\tau}\leq 0.1$ to ensure stability: this
leads to around $\frac{ T_f}{0.1 \times \tau} $
time steps to reach the final time $T_f$.

\end{enumerate}

The scheme we propose here preserves mass but does not preserve
momentum and energy, like most DSMC methods. However, these
quantities are preserved in a statistical way. Moreover at the end of
each time step, the distribution is renormalized to keep the mean
velocity constant and to decrease the error in $\Theta$.

   \subsection{Numerical results}
   \label{subsec:nr}
   We present two different test cases: one for which the correction
   of $\nu$ (section~\ref{sss1}) is not activated and one for
   which it is activated (section~\ref{sss2}). For both cases, the
   characteristic time $\tau$ is set to one, which is sufficient fore
   the illustration given here. We also set $R=1$: in other words, we
   work in non-dimensional variables here.

\subsubsection{First test case} \label{sss1}

We use $1$ million particles. We choose three independent laws for the
three components of velocity of the numerical particles:
\begin{itemize}
\item the first component is equal to $100 s^4-20$ where $s$ follows a
  uniform law between $[0,1]$,
\item the second and third components of the velocity follow a uniform
  law between $[-50,50]$,
\end{itemize}

The choice for the first component seems a little bit strange but we need
to have a non zero heat flux at the beginning of the computations
to be able to capture a characteristic time of variation for
$q$. However, we have chosen distributions whose variances are of
the same order so that the $\nu$ which is used all along this
computation is equal to $-\frac54$. The final time is set to 1.

As expected, we observe the convergence of the directional
temperatures $T_{i,i}$ (the diagonal elements of $\Theta$) towards the
temperature and the relaxation towards the Maxwellian (figure
\ref{f1}). Moreover, all along the computation, the correction of
$\nu$ is not activated and hence the Prandtl number defined
by~(\ref{eq-Pr}) is always equal to $\frac23$ (figure
\ref{f3}). Finally, the Prandtl number is computed by using linear
regression on the logarithm curves of $\Theta(t)$ and $q(t)$ (figure
\ref{f3}), see the discussion at the beginning of
section~\ref{sec:num}: we get a numerical Prandtl number $
Pr_n=\frac{2.8971}{4.5001}=0.6428$, which is close to $\frac23$.

\subsubsection{Second test case}  \label{sss2}

We use $100\,000$  particles. As before we choose three independent laws for the three components of velocity of the numerical particles:
\begin{itemize}
\item the first component is equal to $10\,000x^4-2000$ where $x$ follows a uniform law between $[0,1]$,
\item the second component of velocity follows a uniform law between $[-50,50]$,
\item the third     component of velocity follows a uniform law between $[-50,50]$.
\end{itemize}
We have only changed the first component in order to make active the
correction for $\nu$ by having most of the thermal agitation in one
direction. All the parameters are the same as before except the final
time which now is $0.5$ (we only want to see the change of regime for $\nu$).

\

At time $0.5$ we are not at equilibrium which explains why the
distribution of the first component of velocity is not yet a Gaussian
(figure \ref{f4}). We can still observe the exponential convergence
towards the mean temperature of the diagonal components (figure
\ref{f4}).  When one looks more precisely at the behavior of $\nu$ and
the Prandtl number, one can see that the correction of $\nu$ is
activated so that it varies from $-0.5$ to $-1.25$ and consequently
the Prandtl number varies from $1$ to $\frac23$ (figure
\ref{f5}). One may also see that while $\log |q|$ is still
linear , $\log |T_{11}-T|$ shows a slightly curved profile (figure
\ref{f5}). This is confirmed by the linear fitting of the logarithms
(figure \ref{f6}). However, the numerical Prandtl number computed with
the slope of this linearly fitted lines is then $
Pr_n=\frac{2.948}{4.1127}=0.717$ which is not so far from
$\frac23$. This shows that activating the correction of the
parameter $\nu$ gives correct results.

One may observe that getting $\Pr=1$ at the beginning of the
computation is not correct. However, it has to be noticed that this
happens because the flow is very far from equilibrium for small times:
in such a regime, the Prandtl number has no clear physical
interpretation, since the Chapman-Enskog is not valid, and hence the
transport coefficients are not defined.

   \section{Conclusion}
\label{sec:concl}

In this paper, we have proposed a modified Fokker-Planck model of the
Boltzmann equation that we call ES-FP model. This model satisfies the
properties of conservation of the Boltzmann equation, and has been
defined so as to allow for correct transport coefficients. This has
been illustrated by numerical tests for homogeneous
problems. Moreover, it has been proved that the ES-FP model satisfies
the H-theorem. 

The construction
and analysis of ES-FP models for more complex gases (like polyatomic
or multi-species), as well as
inhomogeneous numerical simulations, will be presented in a
forthcoming paper. Based on our proof of the H-theorem, we believe
this theorem should still be true for such extended models. 

\paragraph{ Acknowledgments.} 
 This study has been carried out in the frame of “the Investments for the future” Programme IdEx Bordeaux – CPU (ANR-10-IDEX-03-02).

 \bibliographystyle{plain}

 \bibliography{biblio}


\begin{figure}[h]
\begin{center}
\includegraphics[height=7cm,angle=270]{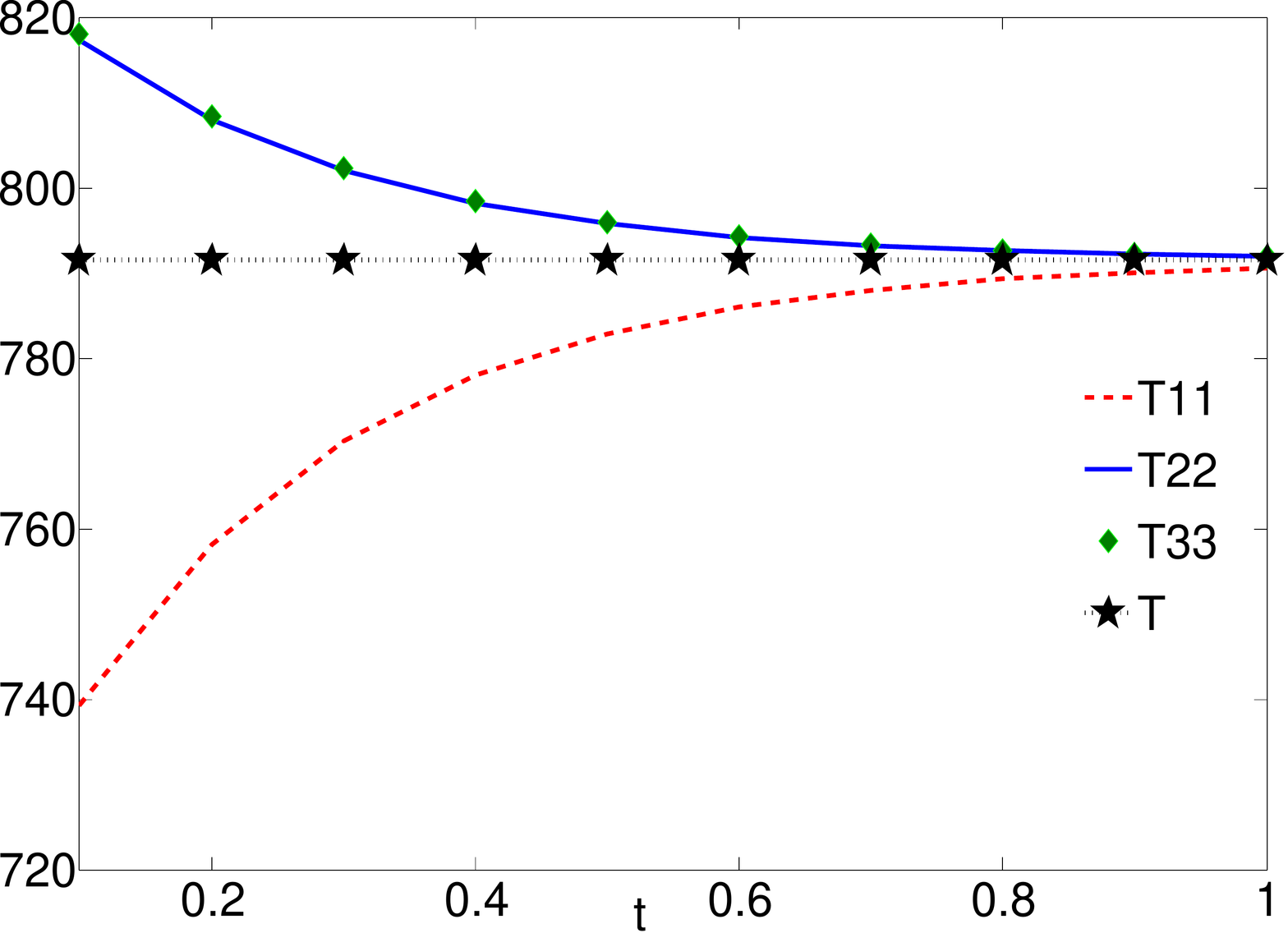}
\includegraphics[height=7cm,angle=270]{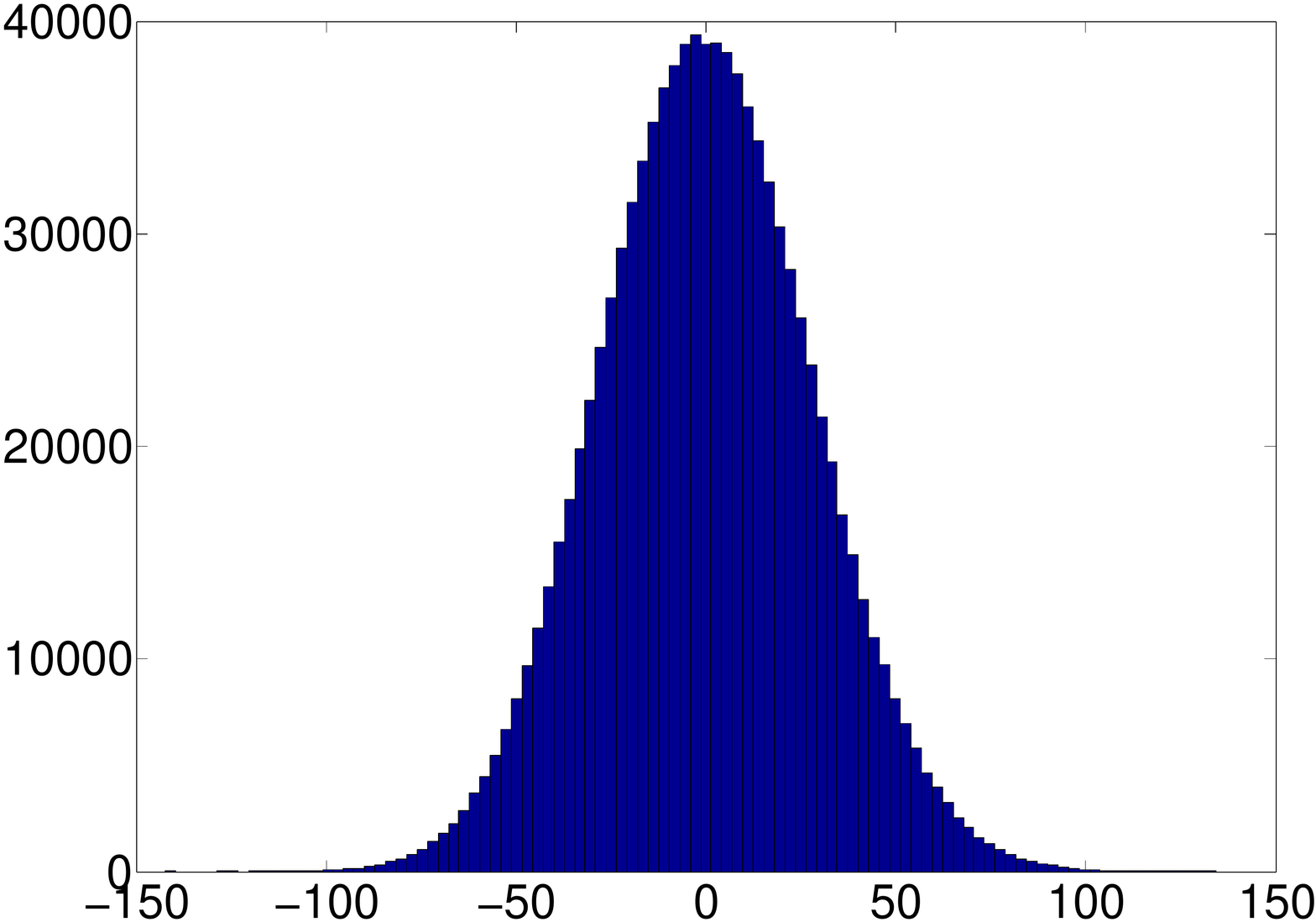}  
\caption{Left: time evolution of the diagonal components
  $T_{11},T_{22},T_{33}$ of the tensor $\Theta$ and of its trace
  $T$. Right: histogram of the first component of velocity at final time
  $t=1$.  \label{f1}}
\end{center}
\end{figure}

\begin{figure}[h]
\begin{center}
\includegraphics[height=7cm,angle=270]{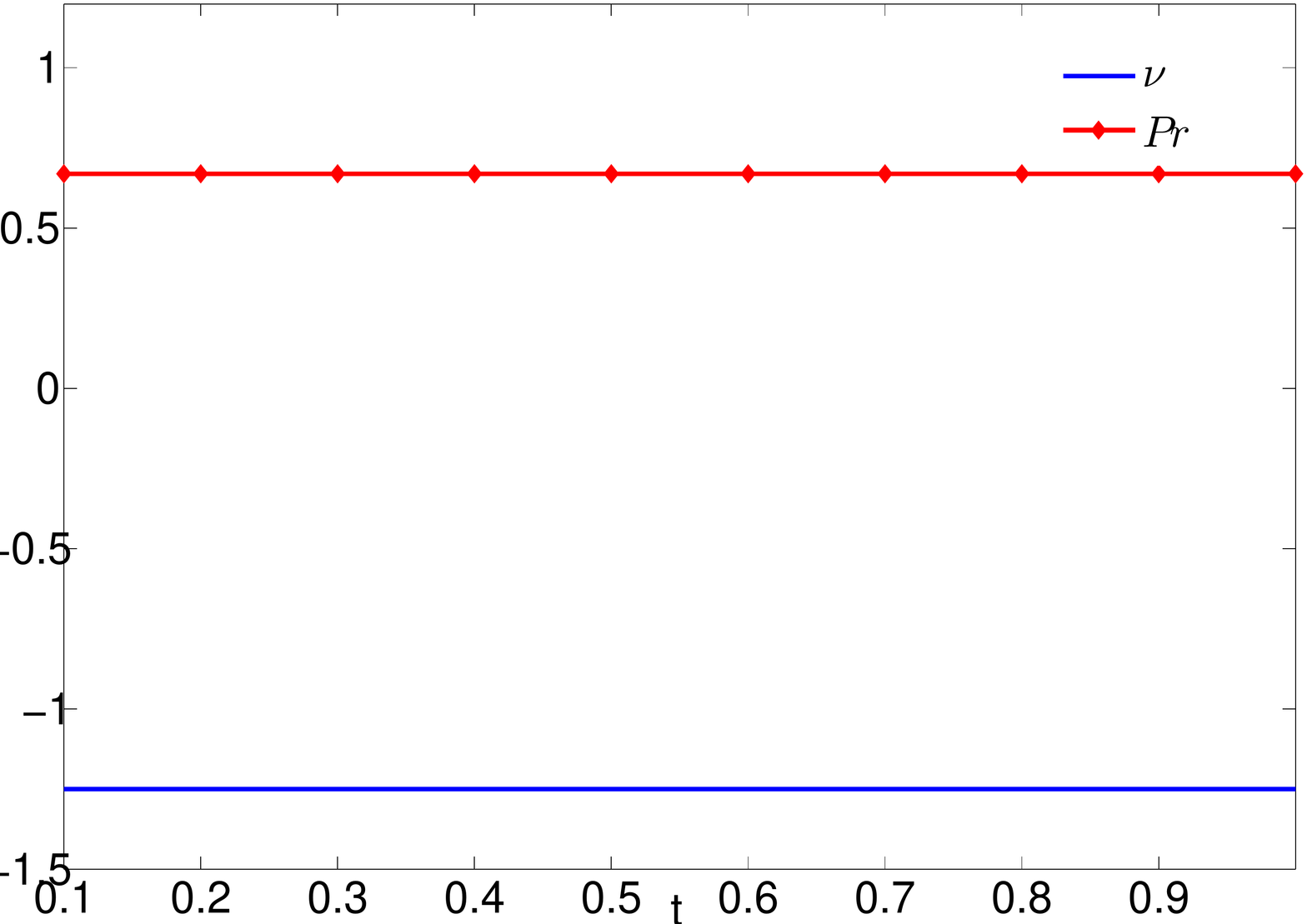} 
\includegraphics[height=7cm,angle=270]{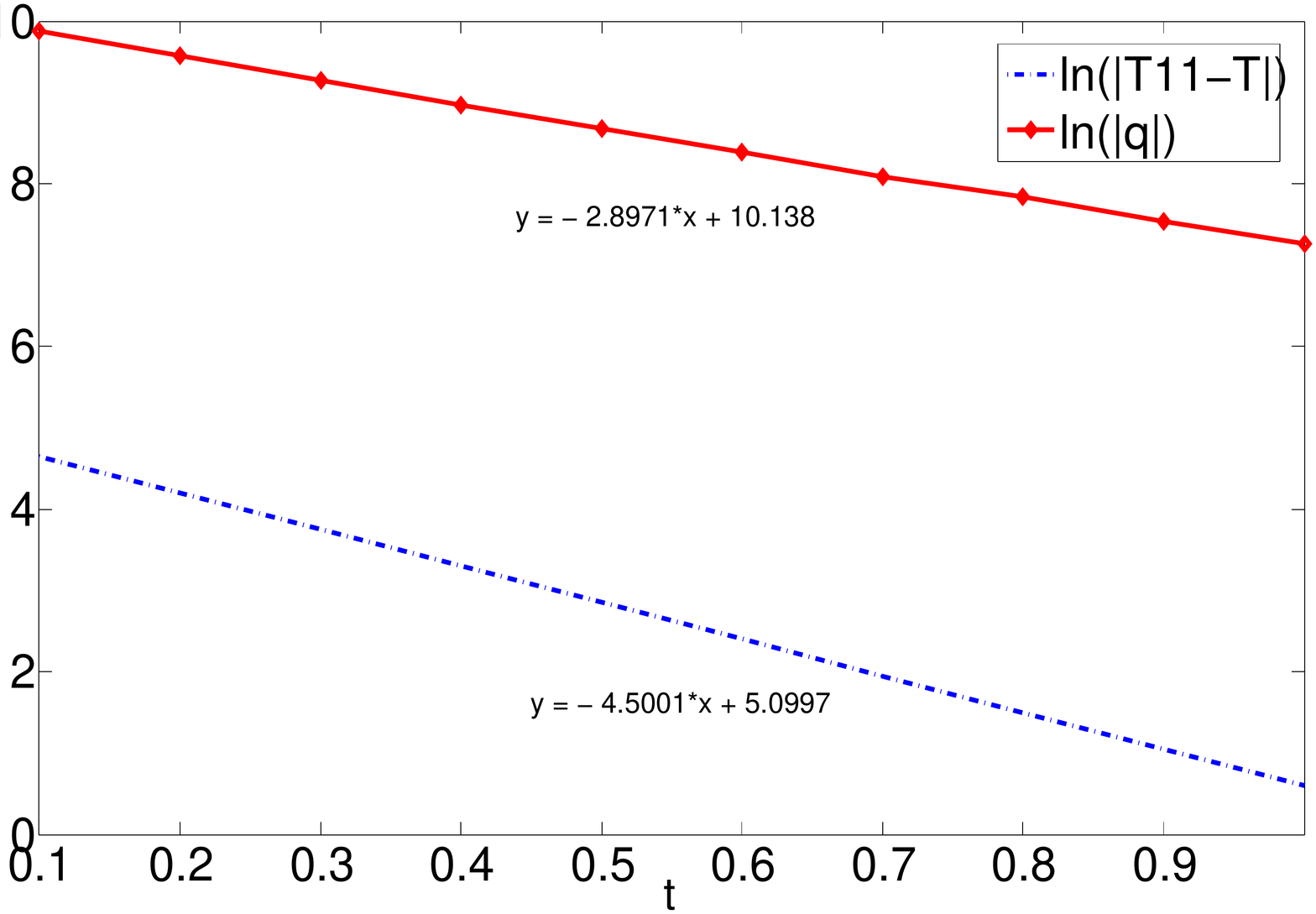}
\caption{Left: time evolution of $\nu$ (defined by~(\ref{eq-nu})) and
  $Pr$ (defined by~(\ref{eq-Pr})). Right: time history of $\log |T-T_{11}|$ and $\log
  |q|$.  \label{f3}}
\end{center}
\end{figure}

\clearpage

\begin{figure}[h]
\begin{center}
\includegraphics[height=7cm,angle=270]{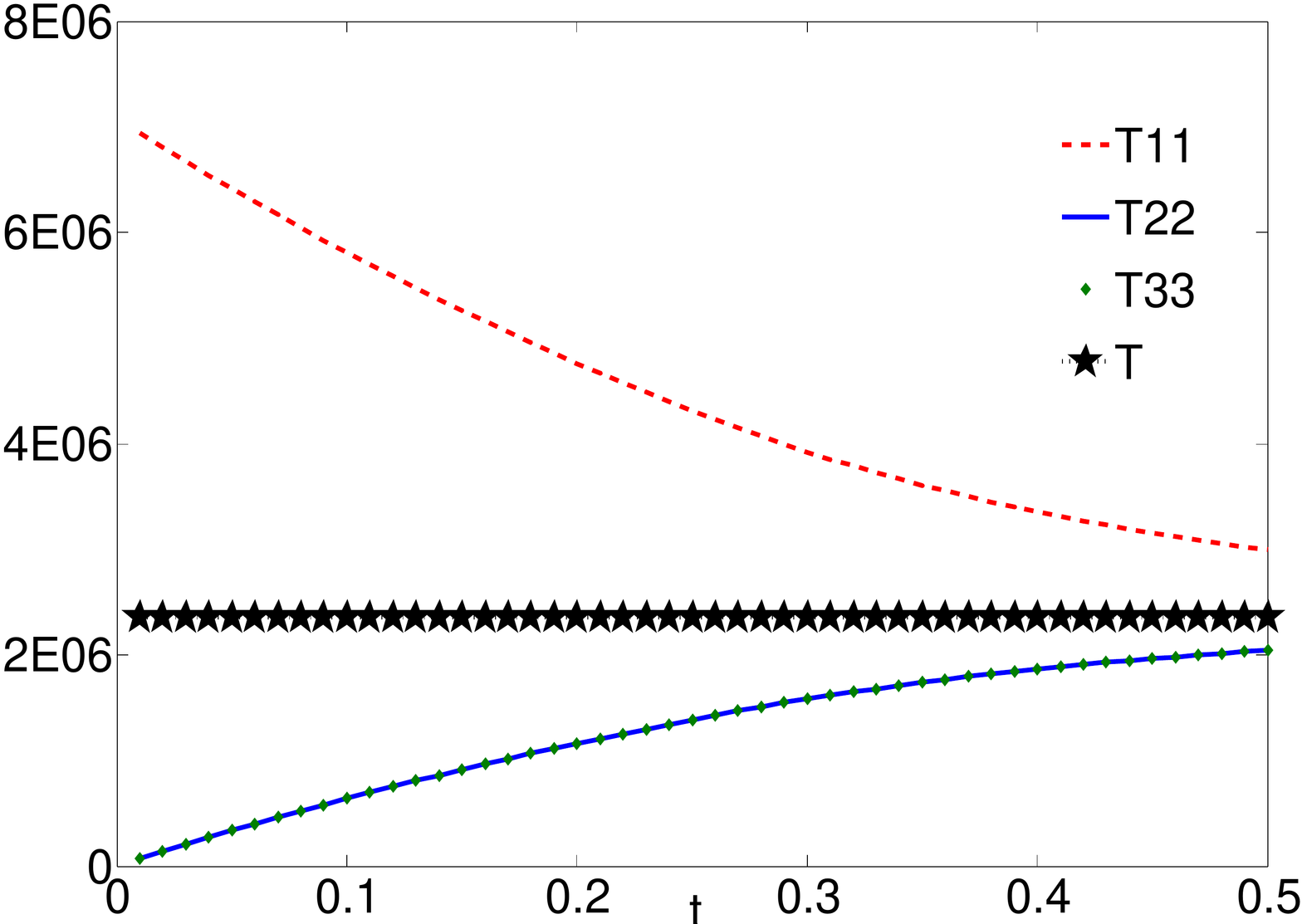}
\includegraphics[height=7cm,angle=270]{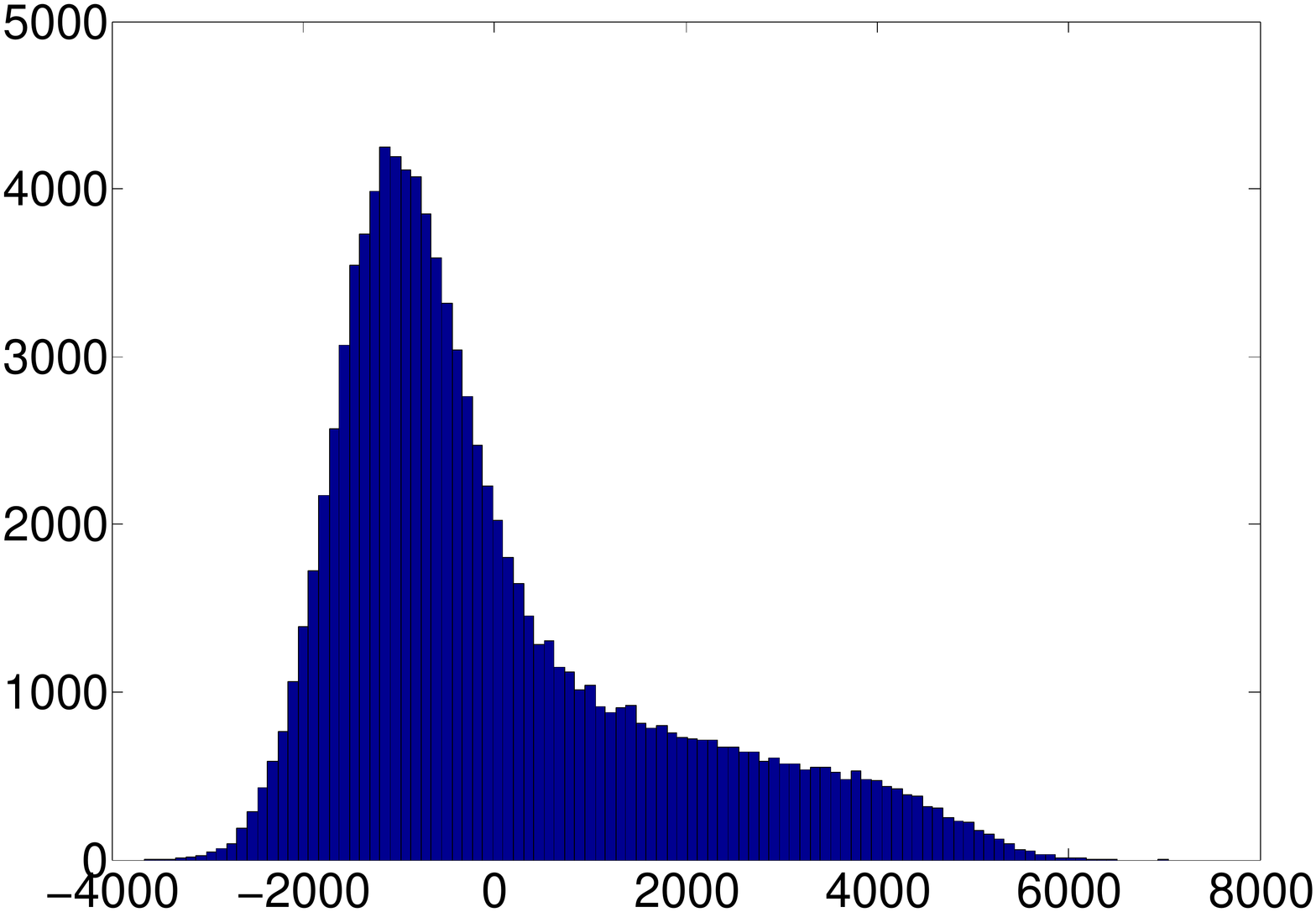}  
\caption{Left: time evolution of the diagonal components
  $T_{11},T_{22},T_{33}$ of the tensor $\Theta$ and of its trace
  $T$. Right: histogram of the first component of the velocity at time $t=0.5$.  \label{f4}}
\end{center}
\end{figure}

\begin{figure}[h]
\begin{center}
\includegraphics[height=8cm,angle=270]{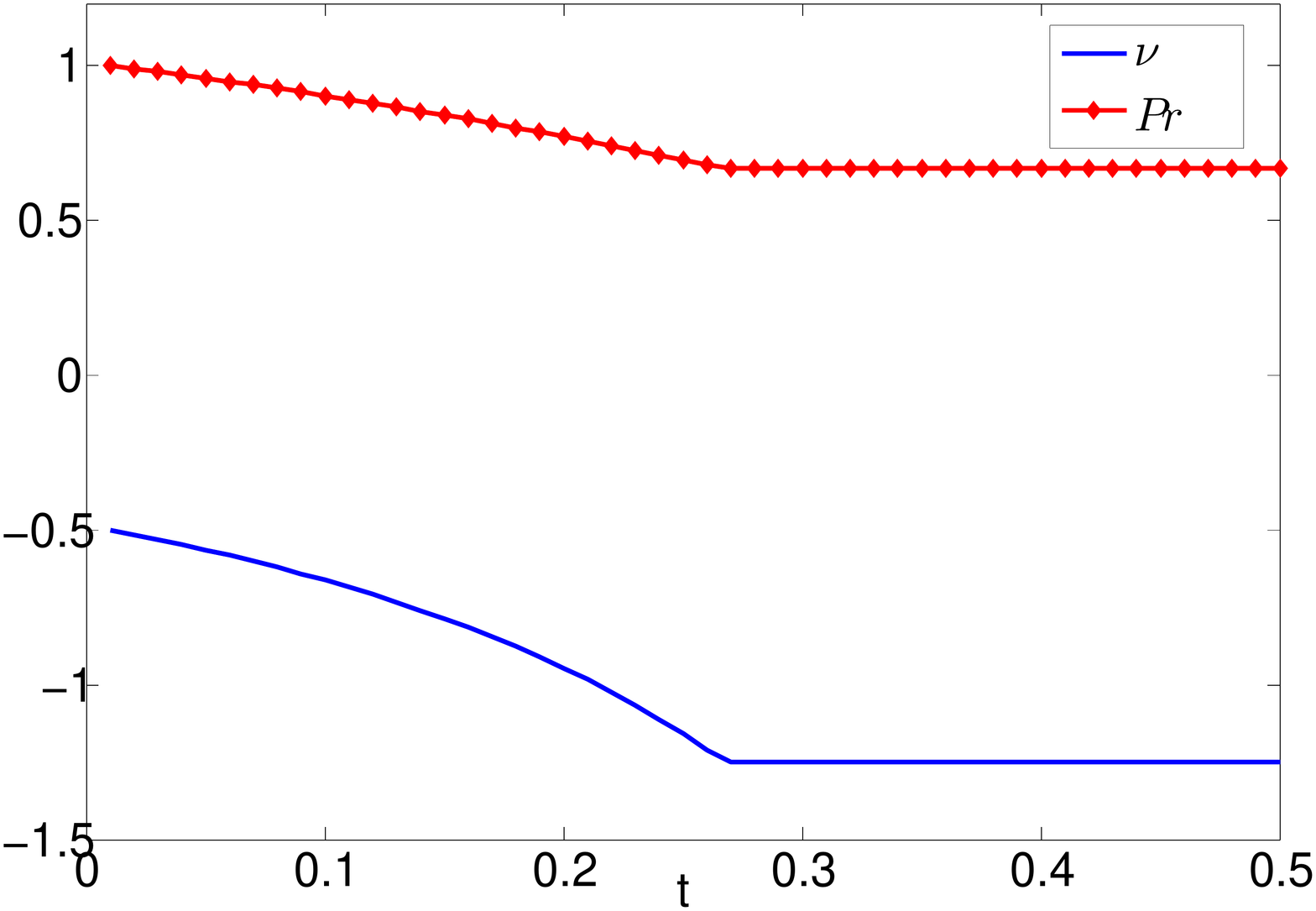} 
\includegraphics[height=8cm,angle=270]{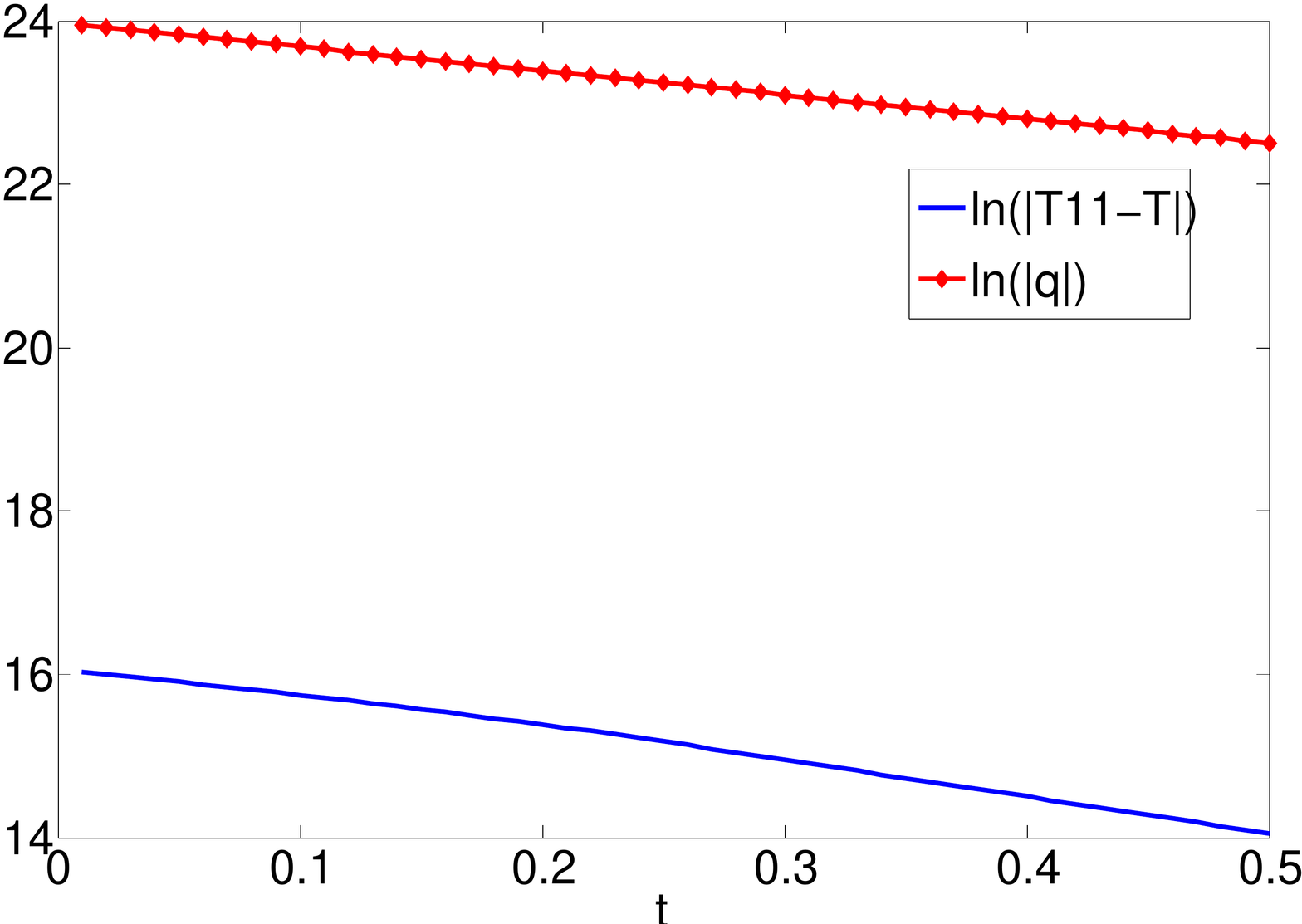}
\caption{Left: $\nu$ and $Pr$ along time. Right: convergence of the logarithms  of  $|T-T_{11}|$ and $q$  \label{f5}}
\end{center}
\end{figure}

\begin{figure}[h]
\begin{center}
\includegraphics[height=7cm,angle=270]{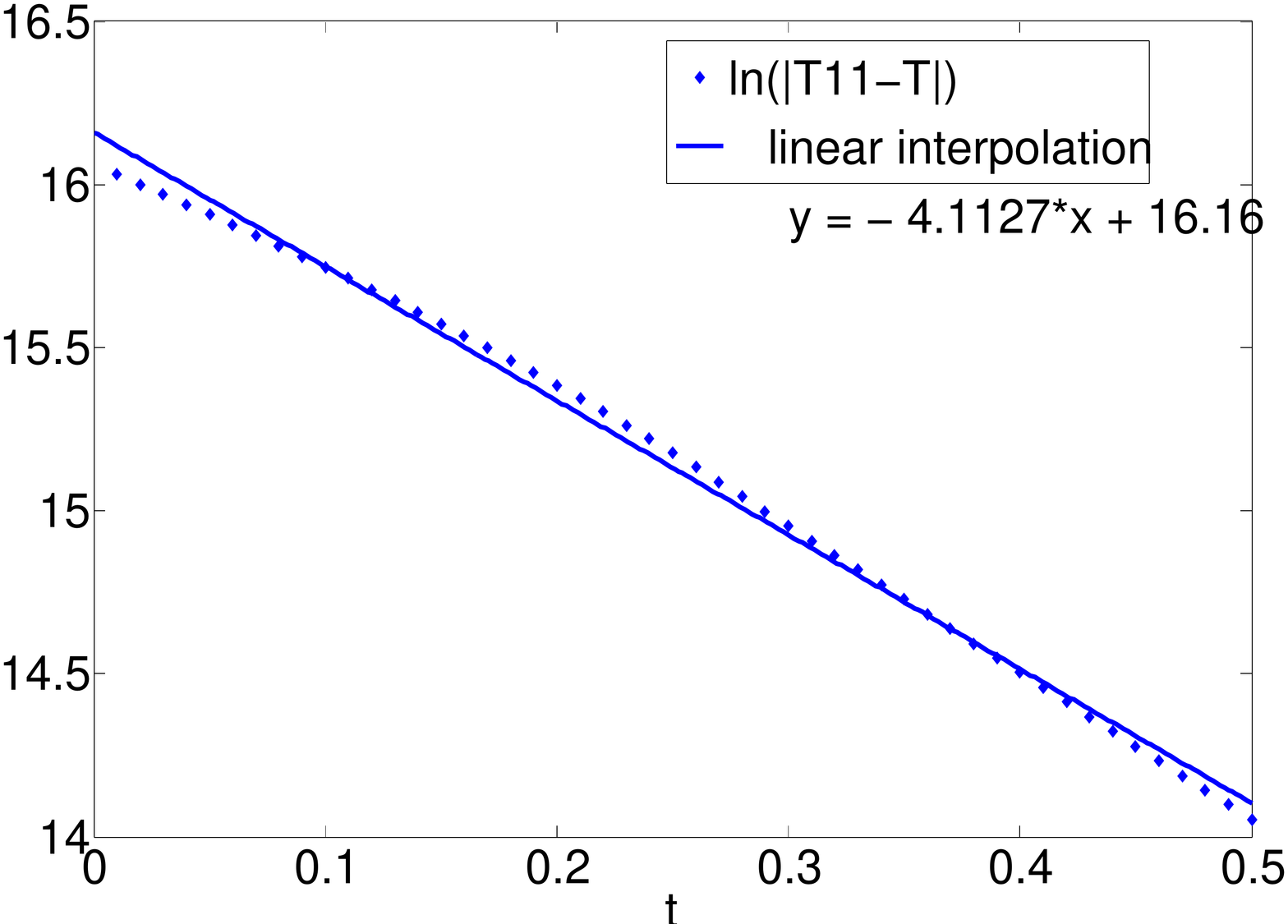}
\includegraphics[height=7cm,angle=270]{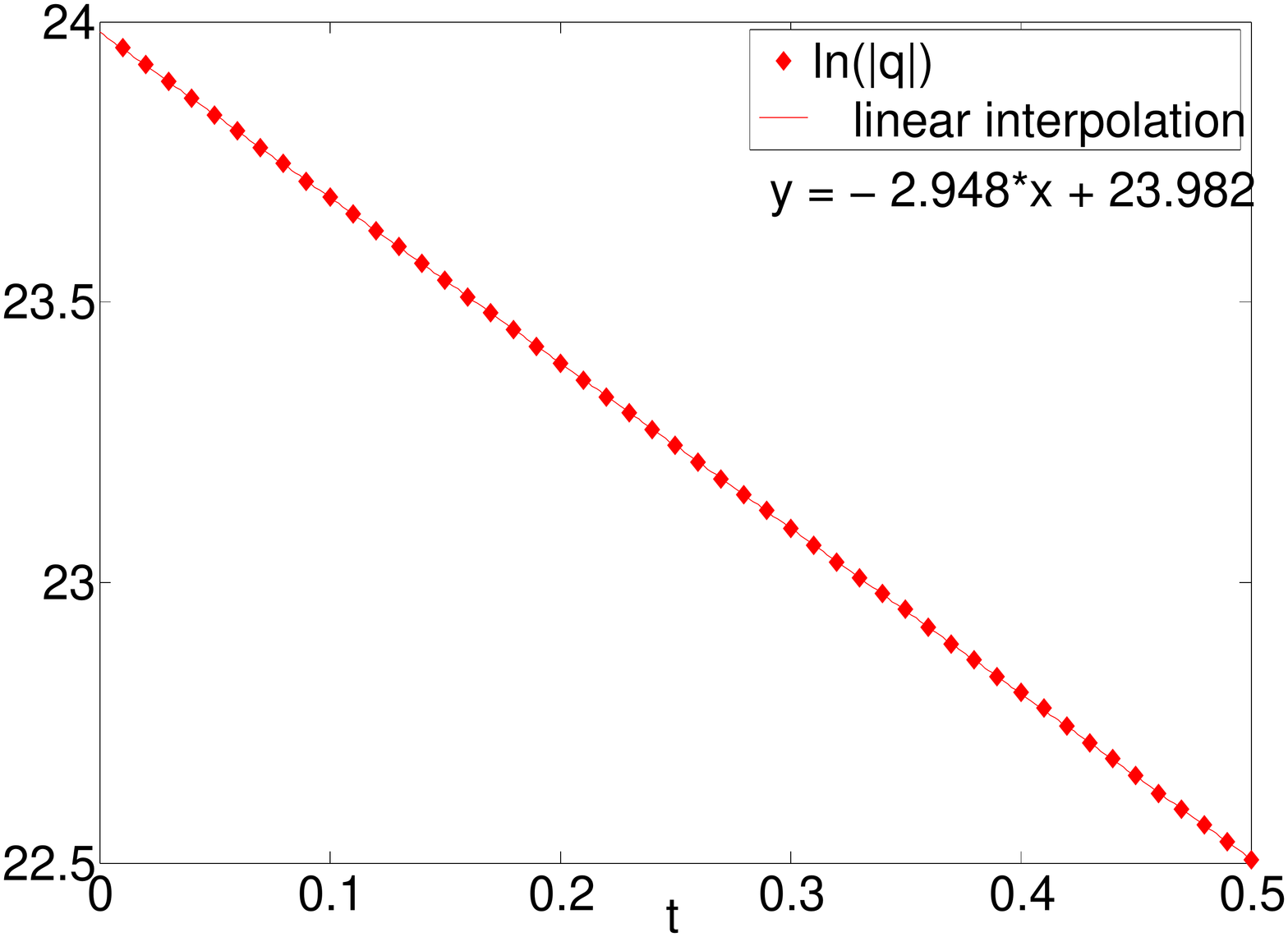}  
\caption{Left: comparison between the curve of the  logarithm  of
  $|T-T_{11}|$ ant its linear fitting. Right: comparison between the curve  of the logarithm  of  $|q|$ ant its linear fitting.  \label{f6}}

\end{center}
\end{figure}

\end{document}